\documentclass{cimart}
\usepackage{xcolor, amstext, amsfonts, amssymb, amsbsy, latexsym}
\usepackage{xy, hhline}
\usepackage{soul}
\usepackage{longtable}
\usepackage{xcolor}
\usepackage{tkz-graph}
\usepackage{url}
\usepackage{float}
\usepackage{tasks}
\usepackage{array}
\usepackage{eurosym}
\usepackage{geometry}
\usepackage{caption,booktabs}
\usepackage{tikz}
\usepackage{cite}
\usepackage{hyperref}

\usepackage[symbol]{footmisc}

\usepackage{verbatim}

\theoremstyle{plain}
 
\newtheorem*{theoremA}{Theorem A}
\newtheorem*{theoremB}{Theorem B}
\newtheorem*{theoremC}{Theorem C}

\title[The geometric classification of nilpotent algebras]{
    The geometric classification of nilpotent  Lie--Yamaguti, Bol
  and compatible Lie algebras\footnote{The first part of the work is supported by FCT UIDB/MAT/00212/2023, UIDP/MAT/00212/2023, and 2023.08952.CEECIND. The second part of this work is supported by grant F-FA-2021-423, the Ministry of Higher Education, Science and Innovations of the Republic of Uzbekistan.}}

\author{
    Kobiljon Abdurasulov, Abror Khudoyberdiyev and Feruza Toshtemirova
    }

\authorinfo[Kobiljon Abdurasulov]{CMA-UBI, University of  Beira Interior, Covilh\~{a}, Portugal;  \  Institute of Mathematics, Academy of Sciences of Uzbekistan, Tashkent, Uzbekistan}{abdurasulov0505@mail.ru}

\authorinfo[Abror Khudoyberdiyev]{Institute of Mathematics, Academy of Sciences of Uzbekistan, Tashkent, Uzbekistan; \ National University of Uzbekistan, Tashkent, Uzbekistan}{khabror@mail.ru}

\authorinfo[Feruza Toshtemirova]{Chirchik State Pedagogical University, Chirchik, Uzbekistan}{feruzaisrailova45@gmail.com}

\abstract{The  geometric classifications of complex $4$-dimensional nilpotent Lie--Yamaguti algebras, $4$-dimensional nilpotent Bol algebras, and $4$-dimensional nilpotent compatible Lie algebras are given.}

\keywords{Lie--Yamaguti algebra, Bol algebra, compatible Lie algebra, nilpotent algebra, rigid algebra, irreducible components, geometric classification.}

\msc{17A30 (primary); 17B63, 14L30 (secondary).}

\VOLUME{33}
\YEAR{2025}
\ISSUE{3}
\NUMBER{10}
\DOI{https://doi.org/10.46298/cm.15810}

\begin{document}

\section*{Introduction}

The study of algebras endowed with two or more multiplication operations is not a recent development; nevertheless, it continues to attract considerable attention within the mathematical community due to its theoretical richness and potential applications \cite{Khr, aae23, b,MF,kms}.
The class of Poisson algebras is the most well-known example of algebras featuring two multiplication operations.

The notion of the Lie triple system was formally introduced as an algebraic object by  Jacobson in connection with problems arising from quantum mechanics. The example of a Lie triple system that arose from quantum mechanics is known as the  Meson field, as described by  Jacobson \cite{Jac}. The notion of Lie--Yamaguti algebras is a  generalization of Lie triple systems and Lie algebras, derived from Nomizu’s work on the invariant affine connections on homogeneous spaces in the 1950s \cite{Nomizu}. The notion of a Lie--Yamaguti algebra is a natural abstraction made by Yamaguti of Nomizu’s considerations. Yamaguti called these systems general Lie triple systems \cite{LYI}.  Kikkawa renamed the notion of a general Lie triple system as Lie triple algebra and showed how to
``integrate'' these algebras to nonassociative multiplications on reductive homogeneous spaces \cite{Kik}.
Kinyon and Weinstein observed that Lie triple
algebras, which they called ``Lie--Yamaguti algebras'' in their paper, can be constructed from Leibniz algebras \cite{LY}.
The interplay between their binary and ternary operations reflects the geometric and algebraic intricacies of the spaces from which they originate, making Lie--Yamaguti algebras a significant object of study in both mathematics and theoretical physics. Lie--Yamaguti algebras have been studied by several authors in \cite{Sagle1, Sagle2, Guo, Goswami, Zhang, LYI,  LYC}. In particular,  Sagle constructed some remarkable examples of Lie--Yamaguti algebras arising from reductive homogeneous spaces in differential geometry. In \cite{Zhang}, the deformation and extension theory of Lie--Yamaguti algebras is studied, and the notion of Nijenhuis operators for Lie--Yamaguti algebras is introduced to describe trivial deformations. The algebraic classification of the nilpotent Lie--Yamaguti algebras up to dimension four was given in \cite{LYC}.

In 1937,   Bol introduced the notion of a left (right) Bol loop \cite{Boldef}, as a loop satisfying the identity $a(b(ac)) = (a(ba))c$, (resp. $((ca)b)a = c((ab)a))$. Recall that a quasigroup is a non-empty set with a binary operation such that each element occurs exactly once in each row and exactly once in each column of the quasigroup multiplication table, and a quasigroup with an identity element is called a loop. The class of Bol algebras plays a similar role concerning Bol loops, as Lie algebras play concerning Lie groups or Malcev
algebras concerning Moufang loops \cite{Maufang}. A connection of the Bol algebras with right-alternative algebras was established in \cite{Mikheev}, and it was demonstrated that the commutator algebra of an arbitrary right-alternative algebra is a Bol algebra. Filippov
classified homogeneous Bol algebras and studied the relation between Bol algebras and Malcev algebras in \cite{Fill}.
Kuz'min and Za\v{\i}di investigated the solvability and semisimplicity of Bol algebras \cite{Kuzmin}. P\'{e}rez--Izquierdo constructed an envelope for Bol algebras and proved that any Bol algebra is located inside the generalized left alternative nucleus of the envelope \cite{Perez}.
Special identities for Bol algebras are investigated in \cite{Irvin}, and it was shown that there are no special identities of degree less than $8$. All the special identities of degree $8$ in the partition six-two were obtained.

A compatible Lie algebra is a pair of Lie algebras such that any linear combination of these two Lie brackets is still a Lie bracket.
Such structures are closely related to the (infinitesimal) deformations of Lie algebras. Compatible Lie algebras are considered in many fields in mathematics and mathematical physics, such as the study of the classical Yang--Baxter equation \cite{ComL1}, integrable equations of the principal chiral model type \cite{ComL2}, elliptic theta functions \cite{Coml3}, and loop algebras over Lie algebras \cite{ComL4}.
Compatible Lie algebras have a natural one-to-one correspondence to compatible linear Poisson structures, whereas the latter are important Poisson structures related to bi-Hamiltonian structures. Recall that a bi-Hamiltonian structure is a pair of Poisson structures on a manifold which are compatible, i.e., their sum is again a Poisson structure \cite{Kosmann}.

Degenerations and geometric properties of a variety of algebras have been an object of study since 1970's. Gabriel~\cite{gabriel} described the irreducible components of the variety of $4$-dimensional unital associative algebras. Mazzola~\cite{maz79} classified, both algebraically and geometrically, the variety of unital associative algebras of dimension $5$.  Burde and Steinhoff~\cite{BC99} constructed the graphs of degenerations for the varieties of $3$-dimensional and $4$-dimensional Lie algebras over $\mathbb{C}$. Grunewald and O'Halloran~\cite{GRH} calculated the degenerations for the nilpotent Lie algebras of dimension up to $5$.

The study of degenerations of algebras is very rich and closely related to deformation theory. It offers an insightful geometric perspective on the subject and has been the object of a lot of research.
One of the main problems of the geometric classification of a variety of algebras is the description of its irreducible components. In the case of finitely many orbits (i.e., isomorphism classes), the irreducible components are determined by the rigid algebras --- algebras whose orbit closure is an irreducible component of the variety under consideration.
There are many results related to the deformation and geometric classification of algebras in different varieties of
associative and non-associative algebras, superalgebras and Poisson algebras, (see \cite{fkkv, afm, BC99, b, lsb20, FKS, k23, l24, MS} and references in \cite{k23, MS}).




In the present paper, we continue the study of the geometric classification of algebras and describe all irreducible components of the variety of four-dimensional nilpotent Lie--Yamaguti algebras, four-dimensional nilpotent Bol algebras, three- and four-dimensional nilpotent compatible Lie algebras. Our main results are summarized below.

\begin{theoremA} The variety of complex four-dimensional nilpotent Lie--Yamaguti algebras has dimension $14$.
It is defined by one rigid algebra and two $1$-parametric families of algebras and can be described as the closure of the union of $\mathrm{GL}_4(\mathbb{C})$-orbits of the algebras given in Theorem \ref{geo1}.
\end{theoremA}

\begin{theoremB} The variety of complex four-dimensional nilpotent Bol algebras has dimension $15$.
It is defined by one rigid algebra and one $1$-parametric family of algebras and can be described as the closure of the union of $\mathrm{GL}_4(\mathbb{C})$-orbits of the algebras given in Theorem \ref{geo2}.
\end{theoremB}


\begin{theoremC}
The variety of complex four-dimensional nilpotent compatible Lie algebras has dimension $13$.
It is defined by one $1$-parametric family of algebras and can be described as the closure of the union of $\mathrm{GL}_4(\mathbb{C})$-orbits of the algebras given in Theorem \ref{geo4}.
\end{theoremC}

\section{Preliminaries: the algebraic classification}

All the algebras below will be over $\mathbb C$ and all the linear maps will be $\mathbb C$-linear.
For simplicity, every time we write the multiplication table of an algebra
the products of basis elements whose values are zero or can be recovered from the commutativity  or from the anticommutativity are omitted.
The notion of a nontrivial algebra means that the  multiplication is nonzero.
In this section, we present the necessary concepts related to Lie--Yamaguti algebras, Bol algebras, and compatible Lie algebras. We present the algebraic classification of four-dimensional nilpotent Lie--Yamaguti algebras, four-dimensional nilpotent Bol algebras.

\begin{definition} Let $\mathcal{L}$ be an algebra with one bilinear multiplication $[\cdot, \cdot]$
     and one trilinear multiplication $[\cdot, \cdot,\cdot]$
     $\big($resp., two bilinear multiplications $[\cdot,\cdot]$ and  $\{\cdot,\cdot\} \big)$. Then
     $\mathfrak D$ is a derivation if it satisfies
\begin{center}
${\mathfrak D}[x,y]=[{\mathfrak D}(x),y]+[x,{\mathfrak D}(y)],\ {\mathfrak D}[x,y,z]=[{\mathfrak D}(x),y,z]+[x,{\mathfrak D}(y),z]+[x,y,{\mathfrak D}(z)].$\\
$\big($resp.,
${\mathfrak D}[x,y]=[{\mathfrak D}(x),y]+[x,{\mathfrak D}(y)],\
{\mathfrak D}\{x,y\}=\{{\mathfrak D}(x),y\}+\{x,{\mathfrak D}(y)\} \big).$
\end{center}
\end{definition}

We denote the set of all derivations of the algebra $\mathcal{L}$ by $\mathfrak{Der}(\mathcal{L})$.

\begin{definition}
     Let $\mathcal{L}$ be an algebra with one bilinear multiplication $[\cdot, \cdot]$
     and one trilinear multiplication $[\cdot, \cdot,\cdot]$   $\big($resp., two bilinear multiplications $[\cdot,\cdot]$ and  $\{\cdot,\cdot\} \big)$. Then
       $\mathcal{L}$ is nilpotent if $\mathcal{L}^{(m)}= 0$, for some $m\geq 2$, where
\begin{center}$\mathcal{L}^{(1)}:= \mathcal{L}, \quad
 \mathcal{L}^{(n)} := \sum\limits_{i+j=n} [\mathcal{L}^{(i)},\mathcal{L}^{(j)}] + \sum\limits_{i+j+k=n+1}[\mathcal{L}^{(i)},\mathcal{L}^{(j)},\mathcal{L}^{(k)}], \
 i,j,k\geq1, n \geq 2.$\\
$\big($resp.,
$\mathcal{L}^{(1)}:= \mathcal{L}, \quad
 \mathcal{L}^{(n)} := \sum\limits_{i+j=n} [\mathcal{L}^{(i)},\mathcal{L}^{(j)}] + \sum\limits_{i+j=n}\{\mathcal{L}^{(i)},\mathcal{L}^{(j)}\}, \
 i,j\geq1, n \geq 2 \big). $\\
 \end{center}

\end{definition}

\begin{definition}[see \cite{LY}] A Lie--Yamaguti algebra is a vector space $\mathcal{L}$  with a bilinear multiplication $[-,-]$ and a trilinear multiplication
$[-,-,-]$ satisfying:

$\begin{array}{ll}
({\rm LY1}) &  [x,y]+[y,x]=0, \\
({\rm LY2}) &  [x,y,z]+[y,x,z]=0, \\
({\rm LY3}) &  [x,y,z]+[y,z,x]+[z,x,y]+[[x,y],z]+[[y,z],x]+[[z,x],y]=0, \\
({\rm LY4}) &  [[x,y],z,u]+[[y,z],x,u]+[[z,x],y,u]=0, \\
({\rm LY5}) &  [x,y,[u,v]]-[[x,y,u],v]-[u,[x,y,v]]=0, \\
({\rm LY6}) &  [u,v,[x,y,z]]-[[u,v,x],y,z]-[x,[u,v,y],z]-[x,y,[u,v,z]]=0.\\
\end{array}$

\end{definition}

\begin{example}

The Lie--Yamaguti algebra with $[x,y]=0$ for any $x, y \in \mathcal{L}$ is exactly
a Lie triple system, closely related to symmetric spaces, while the Lie--Yamaguti
algebra with $[x,y,z]=0$ for any $x, y, z \in \mathcal{L}$ is a Lie algebra.
\end{example}

\begin{theorem}[see \cite{LYC}]
Let $\mathcal{L}$ be a $4$-dimensional complex nilpotent
Lie--Yamaguti algebra. Then $\mathcal{L}$ is isomorphic to one of the following pairwise non-isomorphic algebras:

{\small
\setlength{\tabcolsep}{2pt} 
\begin{longtable}{@{}llllllllll@{}} 

$\mathcal{L}_{01}$ &$:$& $[e_1,e_2]=e_3,$ & $\mbox{(Lie algebra)}.$\\

$\mathcal{L}_{02}$ & $:$&  $[e_1,e_2,e_1]=e_3.$ \\

$\mathcal{L}_{03}$ &$:$& $[e_1,e_2]=e_3,$ & $[e_1,e_2,e_1]=e_3.$\\

$\mathcal{L}_{04}$ &$:$& $[e_1,e_3]=e_4,$ & $[e_2,e_3,e_2]=e_4.$\\

$\mathcal{L}_{05}$ &$:$& $[e_1,e_2]=e_4,$ & $[e_2,e_3,e_2]=e_4.$\\

$\mathcal{L}_{06}$ &$:$& $[e_2,e_3,e_2]=e_4,$ & $[e_3,e_1,e_3]=e_4$.\\

$\mathcal{L}_{07}$ &$:$& $[e_1,e_2]=e_4,$   & $[e_2,e_3,e_2]=e_4,$ & $[e_3,e_1,e_3]=e_4.$\\

$\mathcal{L}_{08}$ &$:$& $[e_1,e_3]=e_4,$   & $[e_2,e_3,e_2]=e_4,$ & $[e_3,e_1,e_3]=e_4.$\\

$\mathcal{L}_{09}$ &$:$& $[e_2,e_3]=e_4,$   & $[e_2,e_3,e_2]=e_4,$ & $[e_3,e_1,e_3]=e_4.$\\

$\mathcal{L}_{10}$ &$:$& $[e_2,e_3,e_1]=e_4,$   & $[e_3,e_1,e_2]=e_4,$ & $[e_2,e_3,e_2]=e_4.$ & $[e_2,e_1,e_3]=2e_4.$\\

$\mathcal{L}_{11}$ &$:$& $[e_1,e_3]=e_4,$   & $[e_2,e_3,e_1]=e_4,$   & $[e_3,e_1,e_2]=e_4,$\\
& & $[e_2,e_3,e_2]=e_4.$ & $[e_2,e_1,e_3]=2e_4.$\\

$\mathcal{L}_{12}$ &$:$& $[e_2,e_3]=e_4,$   & $[e_2,e_3,e_1]=e_4,$   & $[e_3,e_1,e_2]=e_4,$\\
& & $[e_2,e_3,e_2]=e_4.$ & $[e_2,e_1,e_3]=2e_4.$\\

$\mathcal{L}_{13}$ &$:$& $[e_1,e_2]=e_4,$   & $[e_2,e_3,e_1]=e_4,$   & $[e_3,e_1,e_2]=e_4,$ \\
& & $[e_2,e_3,e_2]=e_4.$ & $[e_2,e_1,e_3]=2e_4.$\\

$\mathcal{L}_{14}$ &$:$& $[e_1,e_2]=e_4,$   & $[e_2,e_3]=e_4,$   & $[e_2,e_3,e_1]=e_4,$   & $[e_3,e_1,e_2]=e_4,$ \\
& & $[e_2,e_3,e_2]=e_4.$ & $[e_2,e_1,e_3]=2e_4.$\\

$\mathcal{L}_{15}$ &$:$& $[e_1,e_2]=e_4,$   & $[e_1,e_3]=e_4,$   & $[e_2,e_3,e_1]=e_4,$   & $[e_3,e_1,e_2]=e_4,$ \\
& & $[e_2,e_3,e_2]=e_4,$ & $[e_2,e_1,e_3]=2e_4.$\\

$\mathcal{L}_{16}^{\alpha}$ &$:$&  $[e_2,e_3,e_1]=\alpha e_4,$  & $[e_1,e_2,e_3]=-(1+\alpha)e_4,$   & $[e_3,e_1,e_2]=e_4.$\\

$\mathcal{L}_{17}^{\alpha}$ &$:$& $[e_1,e_2]=e_4,$   & $[e_1,e_2,e_3]=-(1+\alpha)e_4,$   & $[e_2,e_3,e_1]=\alpha e_4,$  & $[e_3,e_1,e_2]=e_4.$\\

$\mathcal{L}_{18}^{\alpha}$ &$:$& $[e_1,e_2]=e_4,$   & $[e_1,e_3]=e_4,$   & $[e_1,e_2,e_3]=-(1+\alpha)e_4,$ \\
& & $[e_2,e_3,e_1]=\alpha e_4,$  & $[e_3,e_1,e_2]=e_4.$\\

$\mathcal{L}_{19}^{\alpha}$ &$:$& $[e_1,e_2]=e_4,$   & $[e_1,e_3]=e_4,$   & $[e_2,e_3]=e_4,$   & $[e_3,e_1,e_2]=e_4.$ \\
& & $[e_2,e_3,e_1]=\alpha e_4,$  & $[e_1,e_2,e_3]=-(1+\alpha)e_4,$ \\

$\mathcal{L}_{20}^{\alpha}$ &$:$& $[e_1,e_2]=e_3,$   & $[e_1,e_3]=e_4,$   &  $[e_1,e_2,e_1]=\alpha e_4.$ \\

$\mathcal{L}_{21}$ &$:$& $[e_1,e_2]=e_3,$   & $[e_2,e_3]=e_4,$   &  $[e_1,e_2,e_1]=e_4.$ \\

$\mathcal{L}_{22}$ &$:$& $[e_1,e_2]=e_3,$   &   $[e_2,e_3,e_2]=e_4.$ \\

$\mathcal{L}_{23}^{\alpha}$ &$:$& $[e_1,e_2]=e_3,$   & $[e_1,e_3]=\alpha e_4,$   & $[e_1,e_2,e_1]=e_4,$  & $[e_2,e_3,e_2]=e_4.$\\

$\mathcal{L}_{24}$ &$:$& $[e_1,e_2]=e_3,$   & $[e_2,e_3]=e_4,$   & $[e_1,e_2,e_1]=e_4,$  & $[e_2,e_3,e_2]=e_4.$\\

$\mathcal{L}_{25}$ &$:$& $[e_1,e_2]=e_3,$   & $[e_1,e_3]=e_4,$   &  $[e_2,e_3,e_2]=e_4.$\\

$\mathcal{L}_{26}$ &$:$& $[e_1,e_2]=e_3,$   & $[e_2,e_3]=e_4,$   &  $[e_2,e_3,e_2]=e_4.$\\

$\mathcal{L}_{27}$ &$:$& $[e_1,e_2]=e_3,$   & $[e_1,e_3,e_2]=e_4,$   &  $[e_2,e_3,e_1]=e_4.$\\

$\mathcal{L}_{28}$ &$:$& $[e_1,e_2]=e_3,$   & $[e_1,e_3]=e_4,$   &$[e_1,e_3,e_2]=e_4,$   &  $[e_2,e_3,e_1]=e_4.$\\

$\mathcal{L}_{29}$ &$:$& $[e_1,e_2]=e_3,$   & $[e_1,e_3]=e_4,$ & $[e_2,e_3]=e_4,$  & $[e_1,e_3,e_2]=e_4,$  \\
& &  $[e_2,e_3,e_1]=e_4.$\\

$\mathcal{L}_{30}$ &$:$&  $[e_1,e_3]=e_4,$ &  $[e_1,e_2,e_1]=e_3.$\\

$\mathcal{L}_{31}$ &$:$& $[e_1,e_2,e_1]=e_3,$   & $[e_1,e_2,e_3]=e_4,$   &  $[e_1,e_3,e_2]=e_4.$\\

$\mathcal{L}_{32}$ &$:$& $[e_1,e_2]=e_4,$   & $[e_1,e_2,e_1]=e_3,$   & $[e_1,e_2,e_3]=e_4,$   &  $[e_1,e_3,e_2]=e_4.$\\

$\mathcal{L}_{33}$ &$:$& $[e_1,e_3]=e_4,$   & $[e_1,e_2,e_1]=e_3,$   & $[e_1,e_2,e_3]=e_4,$   &  $[e_1,e_3,e_2]=e_4.$\\

$\mathcal{L}_{34}$ &$:$& $[e_1,e_2]=e_4,$   & $[e_1,e_3]=e_4,$   & $[e_1,e_2,e_1]=e_3,$   & $[e_1,e_2,e_3]=e_4,$  \\
& &  $[e_1,e_3,e_2]=e_4.$\\

$\mathcal{L}_{35}$ &$:$& $[e_1,e_2,e_1]=e_3,$   & $[e_1,e_3,e_1]=e_4,$   &  $[e_1,e_2,e_2]=e_4.$\\

$\mathcal{L}_{36}$ &$:$& $[e_1,e_3]=e_4,$ & $[e_1,e_2,e_1]=e_3,$   & $[e_1,e_3,e_1]=e_4,$   &  $[e_1,e_2,e_2]=e_4.$\\

$\mathcal{L}_{37}$ &$:$& $[e_1,e_2]=e_4,$ & $[e_1,e_2,e_1]=e_3,$   & $[e_1,e_3,e_1]=e_4,$   &  $[e_1,e_2,e_2]=e_4.$\\

$\mathcal{L}_{38}$ &$:$&  $[e_1,e_2,e_1]=e_3,$   & $[e_1,e_3,e_1]=e_4.$\\

$\mathcal{L}_{39}$ &$:$&  $[e_1,e_2]=e_4,$ & $[e_1,e_2,e_1]=e_3,$   & $[e_1,e_3,e_1]=e_4.$\\

$\mathcal{L}_{40}$ &$:$&  $[e_1,e_3]=e_4,$ & $[e_1,e_2,e_1]=e_3,$   & $[e_1,e_3,e_1]=e_4.$\\

$\mathcal{L}_{41}$ &$:$&  $[e_1,e_3]=e_4,$ & $[e_1,e_2,e_1]=e_3,$   & $[e_1,e_2,e_2]=e_4.$\\

$\mathcal{L}_{42}$ &$:$& $[e_1,e_2]=e_3,$   & $[e_2,e_3]=e_4,$   & $[e_1,e_2,e_1]=e_3,$   & $[e_1,e_2,e_3]=-e_4,$  \\
& &  $[e_1,e_3,e_2]=-e_4.$\\

$\mathcal{L}_{43}$ &$:$& $[e_1,e_2]=e_3,$   & $[e_1,e_3]=e_4,$   & $[e_2,e_3]=e_4,$   & $[e_1,e_2,e_1]=e_3,$\\
& & $[e_1,e_2,e_3]=-e_4,$   &  $[e_1,e_3,e_2]=-e_4.$\\

$\mathcal{L}_{44}^{\alpha}$ &$:$& $[e_1,e_2]=e_3+e_4,$   & $[e_1,e_3]=\alpha e_4,$   & $[e_2,e_3]=e_4,$   & $[e_1,e_2,e_1]=e_3,$\\
& & $[e_1,e_2,e_3]=-e_4,$   &  $[e_1,e_3,e_2]=-e_4.$\\

$\mathcal{L}_{45}^{\alpha}$ &$:$& $[e_1,e_2]=e_3,$   & $[e_1,e_3]=e_4,$   & $[e_1,e_2,e_1]=e_3,$
 & $[e_1,e_3,e_1]=\alpha e_4.$\\

$\mathcal{L}_{46}^{\alpha\neq 0}$ &$:$& $[e_1,e_2]=e_3+\alpha e_4,$   & $[e_1,e_3]=e_4,$   & $[e_1,e_2,e_1]=e_3,$
 & $[e_1,e_3,e_1]=e_4.$\\

$\mathcal{L}_{47}^{\alpha}$ &$:$& $[e_1,e_2]=e_3,$   & $[e_1,e_3]=e_4,$   & $[e_1,e_2,e_1]=e_3,$
 & $[e_1,e_2,e_2]=e_4,$\\
 & &  $[e_1,e_3,e_1]=\alpha e_4.$\\

$\mathcal{L}_{48}$ &$:$& $[e_1,e_2]=e_3,$  & $[e_1,e_2,e_1]=e_3,$
 & $[e_1,e_3,e_1]=e_4.$\\

$\mathcal{L}_{49}$ &$:$& $[e_1,e_2]=e_3,$  & $[e_1,e_2,e_1]=e_3,$
 & $[e_1,e_2,e_2]=e_4,$ & $[e_1,e_3,e_1]=e_4.$\\

$\mathcal{L}_{50}$ &$:$& $[e_1,e_2,e_1]=e_3,$
 & $[e_1,e_2,e_2]=e_4.$\\

$\mathcal{L}_{51}$ &$:$& $[e_1,e_2]=e_4,$
 & $[e_1,e_2,e_1]=e_3.$\\

$\mathcal{L}_{52}$ &$:$& $[e_1,e_2]=e_3,$
 & $[e_1,e_2,e_1]=e_3,$  & $[e_1,e_2,e_2]=e_4.$
\end{longtable} }
Where  $\mathcal{L}_{46}^{0}\cong \mathcal{L}_{45}^{1}.$
\end{theorem}

\begin{remark} The algebras listed below form   Lie triple systems.
$$\mathcal{L}_{02}, \
\mathcal{L}_{06},\
\mathcal{L}_{10},\
\mathcal{L}_{16}^{\alpha},\
\mathcal{L}_{31},\
\mathcal{L}_{35},\
\mathcal{L}_{50}.$$

\end{remark}

\begin{definition}[see \cite{Fill}] A Bol algebra is a vector space $\mathfrak{B}$  with a bilinear multiplication $[-,-]$ and a trilinear multiplication
$[-,-,-]$ satisfying:

$\begin{array}{ll}
({\rm B1}) &  [x,y]+[y,x]=0, \\
({\rm B2}) &  [x,y,z]+[y,x,z]=0, \\
({\rm B3}) &  [x,y,z]+[y,z,x]+[z,x,y]=0, \\
({\rm B4}) &  [u,v,[x,y,z]]-[[u,v,x],y,z]-[x,[u,v,y],z]-[x,y,[u,v,z]]=0, \\
({\rm B5}) &  [[x,y,z],t]-[[x,y,t],z]+[z,t,[x,y]]-[x,y,[z,t]]+[[x,y],[z,t]]=0. \\
\end{array}$

\end{definition}

Identities $({\rm B2})$--$({\rm B4})$ mean that a Bol algebra is a Lie triple system with respect to the trilinear operation $[-,-,-]$. Thus, Lie--Yamaguti algebras and Bol algebras are generalizations of the   Lie triple systems.

\begin{theorem}[see \cite{BC}]
Let $\mathfrak{B}$ be a 4-dimensional complex nilpotent Bol algebra. Then $\mathfrak{B}$ is isomorphic to  one of the
 following Lie--Yamaguti algebras $\mathcal{L}_{01},\dots,\mathcal{L}_{41},$ $\mathcal{L}_{45}^{\alpha},\dots, \mathcal{L}_{52}$ or one of the following mutually non-isomorphic algebras:

{\small
\setlength{\tabcolsep}{2pt} 
\begin{longtable}{@{}llllllllll@{}} 

$\mathfrak{B}_{01}$ &$:$& $[e_1,e_2]=e_3,$   & $[e_1,e_2,e_3]=e_4,$   & $[e_2,e_3,e_1]=-\frac{1}{2}e_4,$ \\
 & & $[e_1,e_3,e_2]=\frac{1}{2}e_4,$ & $[e_2,e_3,e_2]=e_4.$\\

$\mathfrak{B}_{02}$ &$:$& $[e_1,e_2]=e_3,$   & $[e_1,e_3]=e_4,$   & $[e_1,e_2,e_3]=e_4,$   & $[e_2,e_3,e_1]=-\frac{1}{2}e_4,$\\
& & $[e_1,e_3,e_2]=\frac{1}{2}e_4,$ & $[e_2,e_3,e_2]=e_4.$\\

$\mathfrak{B}_{03}$ &$:$& $[e_1,e_2]=e_3,$   & $[e_2,e_3]=e_4,$   & $[e_1,e_2,e_3]=e_4,$   & $[e_2,e_3,e_1]=-\frac{1}{2}e_4,$\\
& & $[e_1,e_3,e_2]=\frac{1}{2}e_4,$ & $[e_2,e_3,e_2]=e_4.$\\

$\mathfrak{B}_{04}^{\alpha}$ &$:$& $[e_1,e_2]=e_3,$   &  $[e_1,e_2,e_3]=e_4,$   & $[e_2,e_3,e_1]=(\alpha-1)e_4,$
&  $[e_1,e_3,e_2]=\alpha e_4.$ \\

$\mathfrak{B}_{05}^{\alpha}$ &$:$& $[e_1,e_2]=e_3,$   & $[e_1,e_3]=e_4,$  &  $[e_1,e_2,e_3]=e_4,$ \\
& & $[e_2,e_3,e_1]=(\alpha-1)e_4,$
&  $[e_1,e_3,e_2]=\alpha e_4.$ \\

$\mathfrak{B}_{06}^{\alpha}$ &$:$& $[e_1,e_2]=e_3,$   & $[e_1,e_3]=e_4,$  & $[e_2,e_3]=e_4,$  & $[e_1,e_2,e_3]=e_4,$ \\
& & $[e_2,e_3,e_1]=(\alpha-1)e_4,$
&  $[e_1,e_3,e_2]=\alpha e_4.$ \\

$\mathfrak{B}_{07}^{\alpha}$ &$:$& $[e_1,e_2]=e_3,$   & $[e_1,e_3]=\alpha e_4,$  &  $[e_1,e_2,e_1]=e_4,$  &  $[e_1,e_2,e_3]=e_4,$ \\
& & $[e_2,e_3,e_1]=e_4,$
&  $[e_1,e_3,e_2]=2e_4.$ \\

$\mathfrak{B}_{08}^{\alpha}$ &$:$& $[e_1,e_2]=e_3,$   & $[e_1,e_3]=\alpha e_4,$  & $[e_2,e_3]=e_4,$  & $[e_1,e_2,e_1]=e_4,$   \\
& &  $[e_1,e_2,e_3]=e_4,$ & $[e_2,e_3,e_1]=e_4,$
&  $[e_1,e_3,e_2]=2e_4.$ \\

$\mathfrak{B}_{09}$ &$:$& $[e_1,e_2]=e_3,$   &  $[e_1,e_2,e_1]=e_3,$  &  $[e_1,e_2,e_3]=e_4,$ & $[e_1,e_3,e_2]=e_4.$ \\

$\mathfrak{B}_{10}$ &$:$& $[e_1,e_2]=e_3,$   & $[e_1,e_3]=e_4,$   &  $[e_1,e_2,e_1]=e_3,$ \\
& &  $[e_1,e_2,e_3]=e_4,$ & $[e_1,e_3,e_2]=e_4.$ \\

\end{longtable} }

\end{theorem}

\begin{definition}
	A compatible Lie algebra is vector space $\mathfrak{g}$ with two bilinear multiplications $[-,-]$ and $\{-,-\}$, where
    $\mathfrak{g}_1=(\mathfrak{g},[-,-])$ and $\mathfrak{g}_2=(\mathfrak{g},\{-,-\})$ are Lie algebras and the two operations are required to satisfy the following identity
    $$\{[x,y],z\}+\{[y,z],x\}+\{[z,x],y\}+[\{x,y\},z]+[\{y,z\},x]+[\{z,x\},y]=0.$$
\end{definition}





\begin{theorem}[see \cite{Comp}]
Let $\mathfrak{g}$ be a complex $3$-dimensional nilpotent
compatible Lie algebra. Then $\mathfrak{g}$ is isomorphic to one of the following pairwise non-isomorphic algebras:

\begin{longtable}{llllllllll}


$
\mathfrak{L}_{1}$ & $:$ &  $ \{e_1,e_2\}= e_3.$   \\

$\mathfrak{L}_{2}^{\alpha}$ & $:$ & $[e_1,e_2]=e_3,$     & $
\{e_1,e_2\}= \alpha e_3.$   \\

\end{longtable}

\end{theorem}

\begin{theorem}[see \cite{Comp}]
Let $\mathfrak{g}$ be a complex $4$-dimensional nilpotent compatible Lie algebra. Then $\mathfrak{g}$ is isomorphic to one of the following pairwise non-isomorphic algebras:

\begin{longtable}{llllllllll}



${L}_{01}$ & $:$ &  $ \{e_1,e_2\}= e_3.$   \\

${L}_{02}^{\alpha}$ & $:$ & $[e_1,e_2]=e_3,$     & $
\{e_1,e_2\}= \alpha e_3.$   \\

${L}_{03}$ &$:$& $[e_2,e_3]=e_4,$     & $
\{e_1,e_3\}= e_4.$   \\





${L}_{04}$ &$:$& $[e_1,e_2]=e_3,$ & $[e_2,e_3]=e_4,$  &  $
\{e_1,e_2\}= e_4.$ \\



${L}_{05}^{\alpha}$ &$:$ & $[e_2,e_3]=\alpha e_4,$   & $
\{e_1,e_2\}= e_3,$  & $
\{e_2,e_3\}= e_4.$ \\

${L}_{06}$ &$:$  & $[e_1,e_3]= e_4,$   & $
\{e_1,e_2\}= e_3,$  & $
\{e_2,e_3\}= e_4.$
\\

${L}_{07}$ &$:$& $[e_1,e_2]=e_4,$     & $
\{e_1,e_2\}= e_3,$  & $
\{e_2,e_3\}= e_4.$ \\


${L}_{08}^{\alpha}$ &$:$& $[e_1,e_2]=e_3,$     & $
\{e_1,e_2\}=\alpha e_3,$  & $
\{e_2,e_3\}= e_4.$ \\

${L}_{09}^{\alpha , \beta}$ &$:$& $[e_1,e_2]=e_3,$   & $[e_2,e_3]=e_4,$   & $
\{e_1,e_2\}=\alpha e_3,$  & $
\{e_2,e_3\}=\beta e_4.$ \\

${L}_{10}^{\alpha}$ &$:$& $[e_1,e_2]=e_3,$   & $[e_2,e_3]=e_4,$   & $
\{e_1,e_2\}=\alpha e_3,$  & $
\{e_1,e_3\}=e_4.$
\end{longtable}

\end{theorem}

\section{The geometric classification}

\subsection{Preliminaries:  geometric classification}\label{refff}

 Given a complex vector space ${\mathbb V}$ of dimension $n$, the set of bilinear and trilinear maps
 \begin{longtable}{rclcl}
 $\textrm{Bil}({\mathbb V} \times {\mathbb V}, {\mathbb V}) $&$\cong $&$\textrm{Hom}({\mathbb V} ^{\otimes2}, {\mathbb V})$&$\cong $&$({\mathbb V}^*)^{\otimes2} \otimes {\mathbb V},$\\
$\textrm{Tril}({\mathbb V} \times {\mathbb V} \times {\mathbb V}, {\mathbb V})$&$ \cong $&$\textrm{Hom}({\mathbb V} ^{\otimes3}, {\mathbb V})$&$\cong$&$ ({\mathbb V}^*)^{\otimes3} \otimes {\mathbb V}$
\end{longtable}\noindent
are vector spaces of dimension $n^3$ and $n^4$, respectively. The set of pairs of bilinear and trilinear maps
\begin{center}$\textrm{Bil}({\mathbb V} \times {\mathbb V}, {\mathbb V}) \oplus \textrm{Tril}({\mathbb V}\times {\mathbb V}\times {\mathbb V}, {\mathbb V}) \cong (({\mathbb V}^*)^{\otimes2} \otimes {\mathbb V}) \oplus (({\mathbb V}^*)^{\otimes3} \otimes {\mathbb V})$ \end{center} is a vector space of dimension $n^4+n^3$. This vector space has the structure of the affine space $\mathbb{C}^{n^4+n^3}$ in the following sense:
fixed a basis $e_1, \ldots, e_n$ of ${\mathbb V}$, any pair $(\mu, \mu') \in \textrm{Bil}({\mathbb V} \times {\mathbb V}, {\mathbb V}) \oplus \textrm{Tril}({\mathbb V}\times {\mathbb V}\times {\mathbb V}, {\mathbb V})$ is determined by some parameters $c_{i,j}^k, c_{i,j,k}^l \in \mathbb{C}$,  called {structural constants},  such that
$$\mu(e_i, e_j) = \sum_{k=1}^n c_{i,j}^k e_k \textrm{ and } \mu'(e_i, e_j,e_k) = \sum_{l=1}^n c_{i,j,k}^l e_l$$
which correspond to a point in the affine space $\mathbb{C}^{n^4+n^3}$. Then a subset $\mathcal S$ of $\textrm{Bil}({\mathbb V} \times {\mathbb V}, {\mathbb V}) \oplus \textrm{Tril}({\mathbb V}\times {\mathbb V}\times {\mathbb V}, {\mathbb V})$ corresponds to an algebraic variety, i.e., a Zariski closed set, if there are some polynomial equations in variables $c_{i,j}^k, c_{i,j,k}^l$ with zero locus equal to the set of structural constants of the pairs in $\mathcal S$.

Given the identities defining Lie--Yamaguti (Bol) algebras, we can obtain a set of polynomial equations in the variables $c_{i,j}^k, c_{i,j,k}^l$. This class of $n$-dimensional Lie--Yamaguti (Bol) algebras is a variety. Denote it by $\mathcal{T}_{n}$.
Now, consider the following action of $\rm{GL}({\mathbb V})$ on ${\mathcal T}_{n}$:
$$(g*\mu)(x,y) := g \mu (g^{-1} x, g^{-1} y), \quad (g* \mu')(x,y,z) :=  g \mu' (g^{-1} x, g^{-1} y, g^{-1} z)$$
for $g\in\rm{GL}({\mathbb V})$, $(\mu, \mu')\in \mathcal{T}_{n}$ and for any $x, y,z \in {\mathbb V}$. Observe that the $\textrm{GL}({\mathbb V})$-orbit of $(\mu, \mu')$, denoted $O((\mu, \mu'))$, contains all the structural constants of the pairs isomorphic to the Lie--Yamaguti (Bol)  algebra with structural constants $(\mu, \mu')$.

A geometric classification of a variety of algebras consists of describing the irreducible components of the variety. Recall that any affine variety can be represented uniquely as a finite union of its irreducible components.
Note that describing the irreducible components of  ${\mathcal{T}_{n}}$ gives us the rigid algebras of the variety, which are those bilinear pairs with an open $\textrm{GL}(\mathbb V)$-orbit. This is because a bilinear pair is rigid in a variety if and only if the closure of its orbit is an irreducible component of the variety.
For this reason, the following notion is convenient. Denote by $\overline{O((\mu, \mu'))}$ the closure of the orbit of $(\mu, \mu')\in{\mathcal{T}_{n}}$.

\begin{definition}
\rm Let ${\rm T} $ and ${\rm T}'$ be two $n$-dimensional Lie--Yamaguti (Bol)  algebras  corresponding to the variety $\mathcal{T}_{n}$, and let $(\mu, \mu'), (\lambda,\lambda') \in \mathcal{T}_{n}$ be their representatives in the affine space, respectively. The algebra ${\rm T}$ is said to {degenerate}  to ${\rm T}'$, and we write ${\rm T} \to {\rm T} '$, if $(\lambda,\lambda')\in\overline{O((\mu, \mu'))}$. If ${\rm T}  \not\cong {\rm T}'$, then we call it a  {proper degeneration}.
Conversely, if $(\lambda,\lambda')\not\in\overline{O((\mu, \mu'))}$ then  we say that ${{\rm T} }$ does not degenerate to ${{\rm T} }'$
and we write ${{\rm T} }\not\to {{\rm T} }'$.
\end{definition}

Furthermore, we have the following notion for a parametric family of algebras.

\begin{definition}
Let ${{\rm T} }(*)=\{{{\rm T} }(\alpha): {\alpha\in I}\}$ be a family of $n$-dimensional Lie--Yamaguti (Bol)  algebras  corresponding to ${\mathcal{T} }_n$ and let ${{\rm T} }'$ be another $n$-dimensional Lie--Yamaguti (Bol)  algebra. Suppose that ${{\rm T} }(\alpha)$ is represented by the structure $(\mu(\alpha),\mu'(\alpha))\in{\mathcal{T} }_n$ for $\alpha\in I$ and ${{\rm T} }'$ is represented by the structure $(\lambda, \lambda')\in{\mathcal{T} }_n$. We say that the family ${{\rm T} }(*)$ {degenerates}   to ${{\rm T} }'$, and write ${{\rm T} }(*)\to {{\rm T} }'$, if $(\lambda,\lambda')\in\overline{\{O((\mu(\alpha),\mu'(\alpha)))\}_{\alpha\in I}}$.
Conversely, if $(\lambda,\lambda')\not\in\overline{\{O((\mu(\alpha),\mu'(\alpha)))\}_{\alpha\in I}}$ then we call it a  {non-degeneration}, and we write ${{\rm T} }(*)\not\to {{\rm T} }'$.

\end{definition}

Observe that ${\rm T}'$ corresponds to an irreducible component of $\mathcal{T}_n$ (more precisely, $\overline{{\rm T}'}$ is an irreducible component) if and only if ${{\rm T} }\not\to {{\rm T} }'$ for any $n$-dimensional Lie--Yamaguti (Bol)  algebra ${\rm T}$ and ${{\rm T}(*) }\not\to {{\rm T} }'$ for any parametric family of $n$-dimensional Lie--Yamaguti (Bol)  algebras ${\rm T}(*)$. In this case, we will use the next ideas to prove that a particular algebra corresponds to an irreducible component.
Firstly, since $\mathrm{dim}\,O((\mu, \mu')) = n^2 - \mathrm{dim}\,\mathfrak{Der}({\rm T})$, it follows that if $ {\rm T} \to  {\rm T} '$ and  ${\rm T} \not\cong  {\rm T} '$, then $\mathrm{dim}\,\mathfrak{Der}( {\rm T} )<\mathrm{dim}\,\mathfrak{Der}( {\rm T} ')$, where $\mathfrak{Der}( {\rm T} )$ denotes the Lie algebra of derivations of  ${\rm T} $.

Secondly,  to show degenerations, let ${{\rm T} }$ and ${{\rm T} }'$ be two Lie--Yamaguti (Bol) algebras represented by the structures $(\mu, \mu')$ and $(\lambda, \lambda')$ from ${{\mathcal T} }_n$, respectively. Let $c_{i,j}^k, c_{i,j,k}^l$ be the structure constants of $(\lambda, \lambda')$ in a basis $e_1,\dots, e_n$ of ${\mathbb V}$. If there exist $n^2$ maps $a_i^j(t): \mathbb{C}^*\to \mathbb{C}$ such that $E_i(t)=\sum_{j=1}^na_i^j(t)e_j$ ($1\leq i \leq n$) form a basis of ${\mathbb V}$ for any $t\in\mathbb{C}^*$ and the structure constants $c_{i,j}^k(t), c_{i,j,k}^l(t)$ of $(\mu, \mu')$ in the basis $E_1(t),\dots, E_n(t)$ satisfy $\lim\limits_{t\to 0}c_{i,j}^k(t)=c_{i,j}^k$ and $\lim\limits_{t\to 0}c_{i,j,k}^l(t)=c_{i,j,k}^l$, then ${{\rm T} }\to {{\rm T} }'$. In this case,  $E_1(t),\dots, E_n(t)$ is called a { \it parametric basis} for ${{\rm T} }\to {{\rm T} }'$.
If $E_{e_1}^t, E_{e_2}^t, E_{e_3}^t, E_{e_4}^t$ is a parametric basis for ${\bf A}\to {\bf B}$, then we denote a degeneration by ${\bf A}\xrightarrow{(E_{e_1}^t, E_{e_2}^t, E_{e_3}^t, E_{e_4}^t)} {\bf B}$.

Thirdly, to prove non-degenerations we may use a remark that follows from the following lemma (see \cite{afm} and the references therein). 

\begin{lemma} 
Consider two Lie--Yamaguti (Bol)  algebras ${\rm T}$ and ${\rm T}'$. Suppose ${\rm T} \to {\rm T}'$. Let ${\mathcal{Z}}$  be a Zariski closed set in ${\mathcal T}_n$ that is stable by the action of the invertible upper (lower) triangular matrices. Then if there is a representation $(\mu, \mu')$ of ${\rm T}$ in ${\mathcal{Z}}$, then there is a representation $(\lambda, \lambda')$ of ${\rm T}'$ in ${\mathcal{Z}}$.
\end{lemma}

To apply this lemma, we will give the explicit definition of the appropriate stable Zariski closed ${\mathcal{Z}}$ in terms of the variables $c_{i,j}^k, c_{i,j,k}^l$ in each case. For clarity, we assume by convention that $c_{i,j}^k=0$ (resp. $c_{i,j,k}^l=0$) if $c_{i,j}^k$ (resp. $c_{i,j,k}^l$) is not explicitly mentioned in the definition of ${\mathcal{Z}}$.

\begin{remark}
\label{redbil}
Let ${{\rm T} }$ and ${{\rm T} }'$ be two Lie--Yamaguti (Bol) algebras represented by the structures $(\mu, \mu')$ and $(\lambda, \lambda')$ from ${{\rm T} }_n$. Suppose ${\rm T}\to{\rm T}'$. Then if $\mu,  \lambda$ represent algebras ${\rm T}_{0}, {\rm T}'_{0}$ in the affine space $\mathbb{C}^{n^3}$ and
$\mu', \lambda'$ represent algebras ${\rm T}_{1}, {\rm T}'_{1}$ in the affine space $\mathbb{C}^{n^4}$
of algebras with a single multiplication, respectively, we have ${\rm T}_{0}\to {\rm T}'_{0}$ and $ {\rm T}_{1}\to {\rm T}'_{1}$.
So, for example, $(0, \mu)$ can not degenerate to $(\lambda, 0)$ unless $\lambda=0$.

\end{remark}

Fourthly, to prove ${{\rm T} }(*)\to {{\rm T} }'$, suppose that ${{\rm T} }(\alpha)$ is represented by the structure $(\mu(\alpha),\mu'(\alpha))\in{\mathcal{T} }_n$ for $\alpha\in I$ and ${{\rm T} }'$ is represented by the structure $(\lambda, \lambda')\in{\mathcal{T} }_n$. Let $c_{i,j}^k, c_{i,j,k}^l$ be the structure constants of $(\lambda, \lambda')$ in a basis  $e_1,\dots, e_n$ of ${\mathbb V}$. If there is a pair of maps $(f, (a_i^j))$, where $f:\mathbb{C}^*\to I$ and $a_i^j:\mathbb{C}^*\to \mathbb{C}$ are such that $E_i(t)=\sum_{j=1}^na_i^j(t)e_j$ ($1\le i\le n$) form a basis of ${\mathbb V}$ for any  $t\in\mathbb{C}^*$ and the structure constants $c_{i,j}^k(t), c_{i,j,k}^l(t)$ of $(\mu\big(f(t)\big),\mu'\big(f(t)\big))$ in the basis $E_1(t),\dots, E_n(t)$ satisfy $\lim\limits_{t\to 0}c_{i,j}^k(t)=c_{i,j}^k$ and $\lim\limits_{t\to 0}c_{i,j,k}^l(t)=c_{i,j,k}^l$, then ${{\rm T} }(*)\to {{\rm T} }'$. In this case  $E_1(t),\dots, E_n(t)$ and $f(t)$ are called a parametrized basis and a parametrized index for ${{\rm T} }(*)\to {{\rm T} }'$, respectively.
Fithly, to prove ${{\rm T} }(*)\not \to {{\rm T} }'$, we may use an analogous of Remark \ref{redbil} for parametric families that follows from Lemma \ref{main2}, see \cite{afm}.

\begin{lemma}\label{main2}
\begin{sloppypar}
Consider the family of Lie--Yamaguti (Bol)  algebras ${\rm T}(*)$ and the Lie--Yamaguti (Bol)  algebra ${\rm T}'$. Suppose ${\rm T}(*) \to {\rm T}'$. Let ${\mathcal{Z}}$  be a Zariski closed set in $\mathcal{T}_n$ that is stable by the action of the invertible upper (lower) triangular matrices. Then if there is a representation $(\mu(\alpha), \mu'(\alpha))$ of ${\rm T}(\alpha)$ in ${\mathcal{Z}}$ for every $\alpha\in I$, then there is a representation $(\lambda, \lambda')$ of ${\rm T}'$ in ${\mathcal{Z}}$.
\end{sloppypar}
\end{lemma}

Finally, the following remark simplifies the geometric problem.

\begin{remark}
 Let $(\mu, \mu')$ and $(\lambda, \lambda')$ represent two Lie--Yamaguti algebras. Suppose $(\lambda, 0)\not\in\overline{O((\mu, 0))}$, $($resp. $(0, \lambda')\not\in\overline{O((0, \mu'))}$$)$, then $(\lambda, \lambda')\not\in\overline{O((\mu, \mu'))}$.
  As we construct the classification of Lie--Yamaguti (Bol)  algebras from a certain class of algebras with a single multiplication which remains unchanged, this remark becomes very useful.
\end{remark}

\subsection{The geometric classification of Lie--Yamaguti   algebras}

In this subsection, we determine all the irreducible components of the variety of four-dimensional nilpotent Lie--Yamaguti   algebras.

\begin{theorem}\label{geo1}
The variety of four-dimensional nilpotent
Lie--Yamaguti algebras has
dimension  $14$ and it has  $3$  irreducible components defined by
\begin{center}
$\mathcal{C}_1=\overline{O({\mathcal L}_{19}^{\alpha})},$ \
$\mathcal{C}_2=\overline{O({\mathcal L}_{29})},$ \ and \
$\mathcal{C}_3=\overline{O({\mathcal L}_{44}^{\alpha})}.$ \

\end{center}
In particular, ${\mathcal L}_{29}$ is a rigid algebra.

\end{theorem}

\begin{proof}
After carefully  checking  the dimensions of orbit closures of the more important algebras, we have

\begin{center}
$\dim  O({\mathcal L}_{19}^{\alpha})=\dim  O({\mathcal L}_{29})=\dim O({\mathcal L}_{44}^{\alpha})=14.$
\end{center}

As the orbit dimensions of these three algebras coincide, no degeneration occurs among them.

The following degenerations are observed:

\begin{longtable}{lcl} \hline

 $\mathcal{L}_{44}^0$ & $\xrightarrow{ (te_1, e_2, te_3+te_4, e_4)}$ & $\mathcal{L}_{01}$ \\ \hline

  $\mathcal{L}_{44}^0$ & $\xrightarrow{ (t^{-1}e_1, t^3e_2, te_3+te_4, e_4)}$ & $\mathcal{L}_{02}$
\\  \hline

$\mathcal{L}_{44}^0$ & $\xrightarrow{ (e_1, te_2, te_3+te_4, e_4)}$ & $\mathcal{L}_{03}$ \\ \hline

$\mathcal{L}_{15}$ & $\xrightarrow{ (t^{2}e_1, te_2, e_3, t^{2}e_4)}$ & $\mathcal{L}_{04}$
\\  \hline

$\mathcal{L}_{15}$ & $\xrightarrow{ (t^{3}e_1, te_2+e_3, t(t+1)e_3, t^{3}(t+1)e_4)}$ & $\mathcal{L}_{05}$

\\  \hline

$\mathcal{L}_{07}$ & $\xrightarrow{ (t^{2}e_1, te_2, e_3, t^{2}e_4)}$ & $\mathcal{L}_{06}$
\\  \hline
$\mathcal{L}_{19}^0$ & $\xrightarrow{ (t^3e_1, -t^2e_1-2t^2e_3, -te_1+te_2-te_3, -2t^5e_4)}$ & $\mathcal{L}_{07}$

 \\  \hline

$\mathcal{L}_{19}^2$ & $\xrightarrow{ (-\frac{5t^2}{64}e_1, \frac{t}{16}e_1+te_3, -\frac{1}{4}e_1+\frac{1}{3}e_2+e_3, -\frac{5t^2}{48}e_4)}$ & $\mathcal{L}_{08}$

 \\  \hline

$\mathcal{L}_{19}^{-1}$ & $\xrightarrow{ (\frac{2t^2}{1-2t}e_1, \frac{2t}{1-2t}e_1+te_3, \frac{1}{1-2t}e_1-e_2+e_3, \frac{2t^2}{2t-1}e_4)}$ & $\mathcal{L}_{09}$

 \\  \hline

$\mathcal{L}_{15}$ & $\xrightarrow{ (t^{-1}e_1, t^{-1}e_2, t^{-1}e_3, t^{-3}e_4)}$ & $\mathcal{L}_{10}$
\\  \hline

$\mathcal{L}_{15}$ & $\xrightarrow{ (e_1, e_2, t^{-1}e_3, t^{-1}e_4)}$ & $\mathcal{L}_{11}$
\\ \hline

 $\mathcal{L}_{15}$ & $\xrightarrow{ (-\frac{1}{2}(t+2)e_2, \frac{(t+1)(t+2)^2}{t^2}e_1+\frac{t^2+3t+2}{t}e_2, -\frac{t+2}{2t}e_3, -\frac{(t+1)(t+2)^3}{2t^3}e_4)}$ & $\mathcal{L}_{12}$

 \\  \hline

$\mathcal{L}_{15}$ & $\xrightarrow{ (t^{-1}e_1, t^{-1}e_2, e_3, t^{-2}e_4)}$ & $\mathcal{L}_{13}$

\\  \hline

 $\mathcal{L}_{15}$ & $\xrightarrow{ (-\frac{1}{2}(t+2)e_2, \frac{(t+2)^2}{t^2}e_1-\frac{t+2}{t}e_2, \frac{t+2}{2}e_3, \frac{(t+2)^3}{2t^2}e_4)}$ & $\mathcal{L}_{14}$

 \\  \hline

 $\mathcal{L}_{19}^{1+t}$ & $\xrightarrow{ (-\frac{2t}{(2+t)^2}e_1-\frac{2t}{(2+t)^2}e_2, -\frac{2}{2+t}e_1+\frac{2}{2+t}e_2, \frac{2}{2+t}e_3, -\frac{8t}{(2+t)^3}e_4)}$ & $\mathcal{L}_{15}$

 \\  \hline

$\mathcal{L}_{19}^{\alpha}$ & $\xrightarrow{ (t^{-1}e_1, t^{-1}e_2, t^{-1}e_3, t^{-3}e_4)}$ & $\mathcal{L}_{16}^{\alpha}$ \\ \hline

$\mathcal{L}_{19}^{\alpha}$ & $\xrightarrow{ (t^{-1}e_1, t^{-1}e_2, e_3, t^{-2}e_4)}$ & $\mathcal{L}_{17}^{\alpha}$
\\  \hline

 $\mathcal{L}_{19}^{\alpha}$ & $\xrightarrow{ (t^{-1}e_1, e_2, e_3, t^{-1}e_4)}$ & $\mathcal{L}_{18}^{\alpha}$ \\  \hline

$\mathcal{L}_{44}^{\frac{1}{t^2-\alpha}}$ & $\xrightarrow{ (te_1, \frac{t^2}{t^2-\alpha }e_2, \frac{t^3}{t^2-\alpha}e_3-\frac{\alpha t^3}{(t^2-\alpha)^2}e_4, \frac{t^4}{(t^2-\alpha)^2}e_4)}$ & $\mathcal{L}_{20}^{\alpha}$

 \\  \hline

 $\mathcal{L}_{44}^{0}$ & $\xrightarrow{ (te_1, -te_2, -t^2e_3-t^2e_4, t^{3}e_4)}$ & $\mathcal{L}_{21}$
\\  \hline

$\mathcal{L}_{23}^{0}$ & $\xrightarrow{ (te_1, e_2, te_3, te_4)}$ & $\mathcal{L}_{22}$ \\ \hline

$\mathcal{L}_{29}$ & $\xrightarrow{ (-\frac{2t^2}{\alpha}e_1-\frac{2t^2}{\alpha^2}e_2, -\frac{t}{\alpha}e_1+\frac{t}{\alpha }e_2+\frac{2t}{\alpha^2}e_3, -\frac{2t^3}{\alpha^2}e_3-\frac{4t^3}{\alpha^3}e_4, \frac{4t^5}{\alpha^4}e_4)}$ & $\mathcal{L}_{23}^{\alpha}$

 \\  \hline

$\mathcal{L}_{23}^{t}$ & $\xrightarrow{ (e_1+t^{-1}e_3, t^{-1}e_1+e_2, e_3-t^{-1}e_4, e_4)}$ & $\mathcal{L}_{24}$

\\  \hline

$\mathcal{L}_{23}^{t^{-1}}$ & $\xrightarrow{ (te_1, e_2, te_3, te_4)}$ & $\mathcal{L}_{25}$  \\ \hline

$\mathcal{L}_{23}^{1}$ & $\xrightarrow{ (te_1+te_3, e_1+e_2, te_3-te_4, te_4)}$ & $\mathcal{L}_{26}$

\\  \hline

$\mathcal{L}_{29}$ & $\xrightarrow{ (t^{-1}e_1, t^{-1}e_2, t^{-2}e_3, t^{-4}e_4)}$ & $\mathcal{L}_{27}$  \\ \hline

$\mathcal{L}_{29}$ & $\xrightarrow{ (t^{-1}e_1, e_2, t^{-1}e_3, t^{-2}e_4)}$ & $\mathcal{L}_{28}$

\\  \hline

$\mathcal{L}_{44}^{1}$ & $\xrightarrow{ (t^{-1}e_1, te_2, t^{-1}e_3, t^{-2}e_4)}$ & $\mathcal{L}_{30}$  \\ \hline $\mathcal{L}_{44}^{0}$ & $\xrightarrow{ (t^{-1}e_1, e_2, t^{-2}e_3, -t^{-3}e_4)}$ & $\mathcal{L}_{31}$

\\  \hline

$\mathcal{L}_{44}^{0}$ & $\xrightarrow{ (t^{-1}e_1, -t^2e_2, -e_3, -te_4)}$ & $\mathcal{L}_{32}$ \\ \hline

$\mathcal{L}_{44}^{-t}$ & $\xrightarrow{ (t^{-1}e_1, te_2, t^{-1}e_3, -t^{-1}e_4)}$ & $\mathcal{L}_{33}$

\\  \hline

$\mathcal{L}_{44}^{t^{2}}$ & $\xrightarrow{ (t^{-1}e_1, -t^2e_2, -e_3, -te_4)}$ & $\mathcal{L}_{34}$  \\ \hline

$\mathcal{L}_{44}^{2t^{3}}$ & $\xrightarrow{ (t^{-1}e_1-2t^2e_2, -2t^3e_2-e_3, -2te_3-2te_4, -4t^2e_4)}$ & $\mathcal{L}_{35}$

  \\  \hline

$\mathcal{L}_{44}^{\frac{2t^4+2t^3}{t-1}}$ & $\xrightarrow{ (\frac{1}{t}e_1+\frac{2t^3}{1-t}e_2, \frac{2t^4}{1-t}e_2+\frac{t}{1-t}e_3, \frac{2t^2}{1-t}e_3-\frac{2t^3}{(1-t)^2}e_4, -\frac{4t^4}{(1-t)^2}e_4)}$ & $\mathcal{L}_{36}$

  \\  \hline
$\mathcal{L}_{44}^{-10t^{3}}$ & $\xrightarrow{ (t^{-1}e_1+2t^2e_2, 2t^3e_2+e_3, 2te_3-2te_4, -4t^2e_4)}$ & $\mathcal{L}_{37}$

  \\  \hline

$\mathcal{L}_{44}^{-t^{2}}$ & $\xrightarrow{ (t^{-1}e_1+te_2, t^2e_2, e_3, -e_4)}$ & $\mathcal{L}_{38}$

\\  \hline

$\mathcal{L}_{44}^{t^{3}}$ & $\xrightarrow{ (t^{-1}e_1-t^2e_2, -t^3e_2, -te_3, -t^2e_4)}$ & $\mathcal{L}_{39}$ \\ \hline

$\mathcal{L}_{44}^{1+t}$ & $\xrightarrow{ (t^{-1}e_1-e_2, -te_2, -t^{-1}e_3, -t^{-2}e_4)}$ & $\mathcal{L}_{40}$

\\  \hline

$\mathcal{L}_{44}^{-2t^3}$ & $\xrightarrow{ (t^{-1}e_1, 2t^4e_2+te_3, 2t^{2}e_3, -4t^{4}e_4)}$ & $\mathcal{L}_{41}$ \\ \hline

$\mathcal{L}_{44}^{0}$ & $\xrightarrow{ (e_1, t^{-1}e_2, t^{-1}e_3, t^{-2}e_4)}$ & $\mathcal{L}_{42}$
\\  \hline

$\mathcal{L}_{44}^{t^{-1}}$ & $\xrightarrow{ (e_1, t^{-1}e_2, t^{-1}e_3, t^{-2}e_4)}$ & $\mathcal{L}_{43}$ \\ \hline

$\mathcal{L}_{44}^{-\frac{1+\alpha}{t}}$ & $\xrightarrow{ (e_1+\frac{\alpha}{t}e_2, e_2, e_3, -t^{-1}e_4)}$ & $\mathcal{L}_{45}^{\alpha}$
\\  \hline

$\mathcal{L}_{44}^{\frac{2}{\alpha}}$ & $\xrightarrow{ (e_1-\frac{1}{\alpha}e_2, -\frac{t}{\alpha}e_2, -\frac{t}{\alpha}e_3, -\frac{t}{\alpha^{2}}e_4)}$ & $\mathcal{L}_{46}^{\alpha\neq0}$

  \\  \hline

$\mathcal{L}_{44}^{\frac{2(\alpha+1)t}{\alpha-1}}$ & $\xrightarrow{ (e_1-\frac{2\alpha t}{\alpha-1}e_2, \frac{2t^2}{1-\alpha}e_2+\frac{t}{1-\alpha}e_3, \frac{2t^2}{1-\alpha}e_3-\frac{2\alpha t^2}{(1-\alpha)^2}e_4, -\frac{4 t^3}{(1-\alpha)^2}e_4)}$ & $\mathcal{L}_{47}^{\alpha\neq 1}$

  \\  \hline

$\mathcal{L}_{44}^{-t^{-1}}$ & $\xrightarrow{ (e_1+t^{-1}e_2, e_2, e_3, -t^{-1}e_4)}$ & $\mathcal{L}_{48}$ \\ \hline

$\mathcal{L}_{44}^{2t}$ & $\xrightarrow{ (e_1-2te_2+e_3, -2t^2e_2-te_3, -2t^2e_3, -4t^{3}e_4)}$ & $\mathcal{L}_{49}$

\\  \hline

$\mathcal{L}_{44}^{0}$ & $\xrightarrow{ (t^{-1}e_1, 2t^2e_2+t^{-1}e_3, 2e_3, -4e_4)}$ & $\mathcal{L}_{50}$

\\  \hline

$\mathcal{L}_{44}^{0}$ & $\xrightarrow{ (t^{-1}e_1, t^3e_2, te_3, t^2e_4)}$ & $\mathcal{L}_{51}$ \\  \hline

$\mathcal{L}_{44}^{0}$ & $\xrightarrow{ (e_1, 2e_2+t^{-1}e_3, 2e_3, -4t^{-1}e_4)}$ & $\mathcal{L}_{52}$

\\  \hline

\end{longtable}

\end{proof}

\subsection{The geometric classification of Bol algebras}

In this section, we determine all the irreducible components of the variety of four-dimensional nilpotent  Bol algebras.

\begin{theorem}\label{geo2}
The variety of four-dimensional nilpotent
Bol algebras has dimension $15$ and it has  $2$  irreducible components defined by
\begin{center}
$\mathcal{C}_1=\overline{O({\mathfrak B}_{06}^{\alpha})}$ \ and \
$\mathcal{C}_2=\overline{O({\mathfrak B}_{10})}.$

\end{center}
In particular, ${\mathfrak B}_{10}$ is a rigid algebra.

\end{theorem}

\begin{proof}

From the proof of Theorem \ref{geo1}, we obtain that there are degenerations from the algebras $\mathcal{L}_{19}^{\alpha}$ and $\mathcal{L}_{29}$ to the algebras $\mathcal{L}_{04}, $ $\mathcal{L}_{05}, $ $\mathcal{L}_{06}, $ $\mathcal{L}_{07},$ $ \mathcal{L}_{08}, $ $\mathcal{L}_{09},$ $ \mathcal{L}_{10}, $ $\mathcal{L}_{11}, $ $\mathcal{L}_{12}, $ $\mathcal{L}_{13},$ $ \mathcal{L}_{14},$ $ \mathcal{L}_{15},$ $ \mathcal{L}_{16}^{\alpha},$ $ \mathcal{L}_{17}^{\alpha}, $ $\mathcal{L}_{18}^{\alpha},$ $ \mathcal{L}_{22}, $ $\mathcal{L}_{23}^{\alpha}, $ $\mathcal{L}_{24},$ $\mathcal{L}_{25},$ $
\mathcal{L}_{26}, $ $\mathcal{L}_{27}, $ $\mathcal{L}_{28}.$

After carefully  checking  the dimensions of orbit closures of the more important  algebras, we have

\begin{center}
$\dim  O({\mathfrak B}_{06}^{\alpha})=15$ \ and \ $\dim  O({\mathfrak B}_{10})=14.$
\end{center}

The following degenerations are observed:

\begin{longtable}{lcl} \hline

 $\mathfrak{B}_{10}$ & $\xrightarrow{ ( te_1, e_2, te_3, e_4)}$ & $\mathcal{L}_{01}$ \\\hline
   $\mathfrak{B}_{10}$ & $\xrightarrow{ (t^{-1}e_1, t^4e_2, t^2e_3, e_4)}$ & $\mathcal{L}_{02}$
\\  \hline

 $\mathfrak{B}_{10}$ & $\xrightarrow{ (e_1,te_2, te_3, e_4)}$ & $\mathcal{L}_{03}$ \\\hline

 $\mathfrak{B}_{06}^{\frac{1}{1+\alpha}}$ & $\xrightarrow{ (-(1+\alpha)e_1, -(1+\alpha)e_2, \frac{1}{t}e_3+\frac{1}{(1+\alpha)t^2}e_4, -\frac{1+\alpha}{t}e_4)}$ & $\mathcal{L}_{19}^{\alpha\neq -1}$

 \\  \hline

 $\mathfrak{B}_{10}$ & $\xrightarrow{ (te_1+t(\alpha+t^2) e_3, te_2, t^2e_3+t^4e_4, t^3e_4)}$ & $\mathcal{L}_{20}^{\alpha}$

\\  \hline

$\mathfrak{B}_{10}$ & $\xrightarrow{ (te_2+te_3, 2te_1+2te_3, -2t^2e_3-2t^2e_4, -4t^3e_4)}$ & $\mathcal{L}_{21}$ \\  \hline

 $\mathfrak{B}_{06}^{\frac{t+1}{t}}$ & $\xrightarrow{ (te_1, te_2, t^2e_3, t^3e_4)}$ & $\mathcal{L}_{29}$

\\  \hline

$\mathfrak{B}_{10}$ & $\xrightarrow{ (t^{-1}e_1, te_2, t^{-1}e_3, t^{-2}e_4)}$ & $\mathcal{L}_{30}$
 \\\hline
 $\mathfrak{B}_{10}$ & $\xrightarrow{ (t^{-1}e_1, t^{-1}e_2, t^{-3}e_3, t^{-5}e_4)}$ & $\mathcal{L}_{31}$
\\  \hline

$\mathfrak{B}_{10}$ & $\xrightarrow{ (t^{-1}e_1-t^{-4}e_3, t^{-1}e_2, t^{-3}e_3-t^{-6}e_4, t^{-5}e_4)}$ & $\mathcal{L}_{32}$ \\\hline
$\mathfrak{B}_{10}$ & $\xrightarrow{ (t^{-1}e_1, e_2, t^{-2}e_3, t^{-3}e_4)}$ & $\mathcal{L}_{33}$

\\  \hline

$\mathfrak{B}_{10}$ & $\xrightarrow{ (t^{-1}e_1-t^{-3}e_3, e_2, t^{-2}e_3-t^{-4}e_4, t^{-3}e_4)}$  &$\mathcal{L}_{34}$\\ \hline

$\mathfrak{B}_{10}$ & $\xrightarrow{ (\frac{1}{t}e_1+\frac{1}{t}e_2, e_2+\frac{1}{2t^3}e_3, \frac{1}{t^2}e_3+\frac{1}{2t^5}e_4, \frac{1}{t^4}e_4)}$ & $\mathcal{L}_{35}$

 \\  \hline

$\mathfrak{B}_{10}$ & $\xrightarrow{ (\frac{1}{t}e_1+e_2+\frac{1-t}{2t^4}e_3,te_2+\frac{1}{2t^2}e_3, \frac{1}{t}e_3+\frac{1}{2t^4}e_4, \frac{1}{t^2}e_4)}$ & $\mathcal{L}_{36}$

 \\  \hline

$\mathfrak{B}_{10}$ & $\xrightarrow{ (\frac{1}{t}e_1+\frac{1}{t}e_2-\frac{1}{t^4}e_3,e_2+\frac{1}{2t^3}e_3, \frac{1}{t^2}e_3-\frac{1}{2t^5}e_4, \frac{1}{t^4}e_4)}$ & $\mathcal{L}_{37}$

 \\  \hline

$\mathfrak{B}_{10}$ & $\xrightarrow{ (\frac{1}{t}e_1+\frac{1}{t}e_2, e_2, \frac{1}{t^2}e_3, \frac{1}{t^4}e_4)}$ & $\mathcal{L}_{38}$\\

\hline

$\mathfrak{B}_{10}$ & $\xrightarrow{ (\frac{1}{t}e_1+\frac{1}{t}e_2-\frac{1}{t^4}e_3,e_2, \frac{1}{t^2}e_3-\frac{1}{t^5}e_4, \frac{1}{t^4}e_4)}$ & $\mathcal{L}_{39}$
 \\  \hline

 $\mathfrak{B}_{10}$ & $\xrightarrow{ (\frac{1}{t}e_1+e_2, te_2, \frac{1}{t}e_3, \frac{1}{t^2}e_4)}$ &$\mathcal{L}_{40}$ \\
   \hline

$\mathfrak{B}_{10}$ & $\xrightarrow{ (\frac{1}{t}e_1+\frac{1}{2t^4}e_3, te_2+\frac{1}{2t^2}e_3, \frac{1}{t}e_3+\frac{1}{2t^4}e_4, \frac{1}{t^2}e_4)}$ & $\mathcal{L}_{41}$\\

\hline $\mathfrak{B}_{10}$ & $\xrightarrow{ (e_1+\alpha e_2, te_2, te_3, te_4)}$ & $\mathcal{L}_{45}$

\\  \hline

$\mathfrak{B}_{10}$ & $\xrightarrow{ (e_1+e_2-\alpha e_3, te_2, te_3-\alpha te_4, te_4)}$ & $\mathcal{L}_{46}^{\alpha\neq0}$

\\\hline

$\mathfrak{B}_{10}$ & $\xrightarrow{ (e_1+\alpha e_2+\frac{1-\alpha}{2t}e_3, te_2+\frac{1}{2}e_3, te_3+\frac{1}{2}e_4, te_4)}$ & $\mathcal{L}_{47}$

 \\  \hline

$\mathfrak{B}_{10}$ & $\xrightarrow{ (e_1+t^{-1}e_2, te_2, te_3, e_4)}$ & $\mathcal{L}_{48}$

\\  \hline

$\mathfrak{B}_{10}$ & $\xrightarrow{ (e_1+\frac{1}{t}e_2+\frac{t-1}{2t^2}e_3, e_2+\frac{1}{2t}e_3, e_3+\frac{1}{2t}e_4, \frac{1}{t}e_4)}$ & $\mathcal{L}_{49}$

 \\  \hline

$\mathfrak{B}_{10}$ & $\xrightarrow{ (\frac{1}{t}e_1+\frac{1}{2t^4}e_3, e_2+\frac{1}{2t^3}e_3, \frac{1}{t^2}e_3+\frac{1}{2t^5}e_4, \frac{1}{t^4}e_4)}$ & $\mathcal{L}_{50}$

 \\  \hline

$\mathfrak{B}_{10}$ & $\xrightarrow{ (e_1+t^{-1}e_3, e_2+t^{-1}e_3, e_3+t^{-1}e_4, 2t^{-1}e_4)}$ & $\mathcal{L}_{52}$

\\  \hline

$\mathfrak{B}_{10}$ & $\xrightarrow{ (t^{-1}e_1+t^{-4}e_3, t^5e_2, t^3e_3+e_4, -te_4)}$ & $\mathcal{L}_{51}$ \\\hline
$\mathfrak{B}_{02}$ & $\xrightarrow{ (t^{-1}e_1, t^
{-1}e_2, t^{-2}e_3, t^{-4}e_4)}$ & $\mathfrak{B}_{01}$

\\  \hline

$\mathfrak{B}_{06}^{\frac{1+2t}{2}}$ & $\xrightarrow{ (-2te_2, -e_1+e_2, -2te_3, 4t^2e_4)}$ & $\mathfrak{B}_{02}$ \\\hline

$\mathfrak{B}_{02}$ & $\xrightarrow{ (t^{-1}e_1, t^
{-2}e_1+t^{-1}e_2, t^{-2}e_3, t^{-4}e_4)}$ & $\mathfrak{B}_{03}$

\\  \hline

$\mathfrak{B}_{06}^{\alpha} $ & $\xrightarrow{ (t^{-1}e_1, t^{-1}e_2, t^{-2}e_3, t^{-4}e_4)}$ & $\mathfrak{B}_{04}^{\alpha}$ \\ \hline
$\mathfrak{B}_{06}^{\alpha} $ & $\xrightarrow{ (t^{-1}e_1, e_2, t^{-1}e_3, t^{-2}e_4)}$ & $\mathfrak{B}_{05}^{\alpha}$ \\ \hline

$\mathfrak{B}_{06}^{2+t}$ & $\xrightarrow{ (\frac{1}{t}e_1-\frac{1}{\alpha t^2}e_3, \frac{1}{\alpha}e_2, \frac{1}{\alpha t}e_3+\frac{1}{\alpha^2t^2}e_4, \frac{1}{\alpha^2t^2}e_4)}$ & $\mathfrak{B}_{07}^{\alpha\neq0}$

\\  \hline

$\mathfrak{B}_{06}^{2+t}$ & $\xrightarrow{ (e_1-\frac{1}{\alpha t}e_3, \frac{1}{\alpha}e_2, \frac{1}{\alpha}e_3+\frac{1}{\alpha^2t}e_4, \frac{1}{\alpha^2}e_4)}$ & $\mathfrak{B}_{08}^{\alpha\neq0}$

\\  \hline
$\mathfrak{B}_{10}$ & $\xrightarrow{ (e_1, t^{-1}e_2, t^{-1}e_3, t^{-2}e_4)}$ & $\mathfrak{B}_{09}$

\\  \hline

\end{longtable}

Below we list all important reasons for necessary non-degeneration.

\begin{longtable}{lcl|l}
\hline
    \multicolumn{4}{c}{Non-degenerations reasons} \\
\hline

$\mathfrak{B}_{06}^{\alpha}$ & $\not \rightarrow  $ &

$\mathfrak{B}_{10}$
&
$\mathcal R=\left\{
c_{1,2,1}^{3}=0, \ c_{2,3,2}^{4}=0
\right\}
$\\
\hline
\end{longtable}
Here, the coefficients $c_{i,j}^{k}$ and $c_{i,j,k}^{l}$  are structural constants in the basis $\{e_1,e_2,e_3,e_4\}$.
\end{proof}

\subsection{The geometric classification of  compatible Lie algebras}



In this subsection, we determine all the irreducible components of the variety of three and four-dimensional nilpotent compatible Lie algebras.

The notions of orbit closure, degeneration, and irreducible components for algebras with two bilinear multiplications are defined analogously to those in Section \ref{refff}; see, for example, \cite{PreLie}. The following results establish that the varieties of three- and four-dimensional nilpotent compatible Lie algebras each consist of a single irreducible component

\begin{proposition}
The variety of three-dimensional nilpotent compatible Lie algebras
is irreducible and defined by
$\mathcal{C}=\overline{O(\mathfrak{L}_{2}^{\alpha})}$, and the dimension of this variety is equal to $4$.

\end{proposition}

\begin{proof}

After carefully  checking  the dimensions of orbit closures of three-dimensional nilpotent compatible Lie algebras, we have
$\dim  O(\mathfrak{L}_{2}^{\alpha})=4.$
It is not difficult to verify the degeneration
$\mathfrak{L}_{2}^{t^{-1}}\ \xrightarrow{ (te_1, e_2, e_3)} \ \mathfrak{L}_{1}.$
Therefore, we get that $\mathcal{C}=\overline{O(\mathfrak{L}_{2}^{\alpha})}.$
\end{proof}

\begin{theorem}\label{geo4}
The variety of four-dimensional nilpotent
compatible Lie algebras
is irreducible and defined by
$\mathcal{C}=\overline{O({L_{10}^{\alpha})}}$, and the dimension of this variety is equal to $13$.

\end{theorem}

\begin{proof}

After carefully  checking  the dimensions of orbit closures of the more important  algebras, we have $\dim  O(L_{10}^{\alpha})=13.$


The following degenerations are observed:

\begin{longtable}{lcl} \hline
$L_{10}^{t^{-1}}$ & $\xrightarrow{ (  e_1, t^2e_2,  te_3 , e_4)}$ & $L_{01}$  \\\hline
$L_{10}^{\alpha}$ & $\xrightarrow{ ( t e_1, e_2, t e_3 , e_4)}$ & $L_{02}^{\alpha}$
\\  \hline

$L_{10}^{0}$ & $\xrightarrow{ (t e_1, t e_2, e_3, te_4)}$ & $L_{03}$ \\\hline
$L_{10}^{0}$ & $\xrightarrow{ (te_1, e_2, te_3, te_4)}$ & $L_{04}$

\\  \hline

$L_{10}^{t^{-1}}$ & $\xrightarrow{ ( te_2, e_1 + \alpha e_2,  -e_3,  -e_4)}$ & $L_{05}^{\alpha}$ \\\hline
$L_{10}^{t^{-1}}$ & $\xrightarrow{ (e_2, e_1, -t^{-1}e_3, -t^{-1}e_4)}$ & $L_{06}$

\\  \hline

$L_{10}^{t^{-1}}$ & $\xrightarrow{ ( t^2e_2, te_1-te_3, -t^2e_3, -t^3e_4)}$ & $L_{07}$ \\\hline
$L_{10}^{\alpha}$ & $\xrightarrow{ ( te_2, e_1, -te_3, -te_4)}$ & $L_{08}^{\alpha}$
\\  \hline

$L_{10}^{\alpha}$ & $\xrightarrow{ ( te_2,  \beta e_1+e_2  , -\beta t e_3 ,  - \beta t e_4)}$  &$L_{09}^{\alpha,\beta\neq0}$\\

 \hline

\end{longtable}

Hence, the variety of four-dimensional nilpotent compatible Lie algebras has a single irreducible component $\overline{O({L_{10}^{\alpha})}}.$
\end{proof}



\EditInfo{June 9, 2025}{August 2, 2025}{David Towers and Ivan Kaygorodov}


{\small

    \begin{thebibliography}{10}
    
    \bibitem{LYC}
    H.~Abdelwahab, E.~Barreiro, A.~Calderón, and A.~Fernández~Ouaridi.
    \newblock The classification of nilpotent {L}ie-{Y}amaguti algebras.
    \newblock {\em Linear Algebra and its Applications}, 654:339--378, 2022.
    
    \bibitem{BC}
    H.~Abdelwahab and A.~Calderón.
    \newblock The classification of nilpotent {B}ol algebras.
    \newblock {\em Communications in Algebra}, 53(5):1--14, 2024.
    
    \bibitem{fkkv}
    H.~Abdelwahab, A.~Fern{á}ndez~Ouaridi, and I.~Kaygorodov.
    \newblock Degenerations of {P}oisson-type algebras.
    \newblock {\em Rendiconti del Circolo Matematico di Palermo}, 74(1):63, 2025.
    
    \bibitem{afm}
    H.~Abdelwahab, A.~Fern{á}ndez~Ouaridi, and C.~Mart{í}n~Gonz{á}lez.
    \newblock Degenerations of {P}oisson algebras.
    \newblock {\em Journal of Algebra and Its Applications}, 24(3):2550087, 2025.
    
    \bibitem{PreLie}
    H.~Abdelwahab, I.~Kaygorodov, and A.~Makhlouf.
    \newblock The algebraic and geometric classification of compatible pre-{L}ie
      algebras.
    \newblock {\em Symmetry, Integrability and Geometry: Methods and Applications},
      20(107):20, 2024.
    
    \bibitem{aae23}
    K.~Abdurasulov, J.~Adashev, and S.~Eshmeteva.
    \newblock Transposed Poisson structures on solvable {L}ie algebras with
      filiform nilradical.
    \newblock {\em Communications in Mathematics}, 32(3):441--483, 2024.
    
    \bibitem{b}
    A.~Ben~Hassine, T.~Chtioui, M.~Elhamdadi, and S.~Mabrouk.
    \newblock Cohomology and deformations of left-symmetric {R}inehart algebras.
    \newblock {\em Communications in Mathematics}, 32(2):127--152, 2024.
    
    \bibitem{Boldef}
    G.~Bol.
    \newblock Gewebe und gruppen.
    \newblock {\em Mathematische Annalen}, 114(1):414–431, 1937.
    
    \bibitem{BC99}
    D.~Burde and C.~Steinhoff.
    \newblock Classification of orbit closures of $4$-dimensional complex {L}ie
      algebras.
    \newblock {\em Journal of Algebra}, 214(2):729--739, 1999.
    
    \bibitem{MF}
    M.~Fairon.
    \newblock Modified double brackets and a conjecture of {S}. {A}rthamonov.
    \newblock {\em Communications in Mathematics}, 33(3):5, 2025.
    
    \bibitem{FKS}
    R.~Fehlberg~J\'{u}nior, I.~Kaygorodov, and A.~Saydaliyev.
    \newblock The geometric classification of symmetric {L}eibniz algebras.
    \newblock {\em Communications in Mathematics}, 33(1):10, 2025.
    
    \bibitem{Fill}
    V.~T. Filippov.
    \newblock Homogeneous {B}ol algebras.
    \newblock {\em Siberian Mathematical Journal}, 35(4):818--825, 1994.
    
    \bibitem{gabriel}
    P.~Gabriel.
    \newblock Finite representation type is open.
    \newblock {\em Proceedings of the International Conference on Representations
      of Algebras. Carleton University, Ottawa}, pages 132--155, 1974.
    
    \bibitem{ComL2}
    I.~Z. Golubchik and V.~V. Sokolov.
    \newblock Compatible {L}ie brackets and integrable equations of the principal
      chiral field model type.
    \newblock {\em Functional Analysis and Its Applications}, 36(3):172--181, 2002.
    
    \bibitem{ComL4}
    I.~Z. Golubchik and V.~V. Sokolov.
    \newblock Factorization of the loop algebras and compatible {L}ie brackets.
    \newblock {\em Journal of Nonlinear Mathematical Physics}, 12:343--350, 2005.
    
    \bibitem{ComL1}
    I.~Z. Golubchik and V.~V. Sokolov.
    \newblock Compatible {L}ie brackets and the {Y}ang--{B}axter equation.
    \newblock {\em Theoretical and Mathematical Physics}, 146(2):159--169, 2006.
    
    \bibitem{Goswami}
    S.~Goswami, S.~K. Mishra, and G.~Mukherjee.
    \newblock Automorphisms of extensions of {L}ie-{Y}amaguti algebras and
      inducibility problem.
    \newblock {\em Journal of Algebra}, 641:268--306, 2024.
    
    \bibitem{GRH}
    F.~Grunewald and J.~O'Halloran.
    \newblock Varieties of nilpotent {L}ie algebras of dimension less than six.
    \newblock {\em Journal of Algebra}, 112:315--325, 1988.
    
    \bibitem{Guo}
    {\relax Sh}.~Guo, B.~Mondal, and R.~Saha.
    \newblock On equivariant {L}ie-{Y}amaguti algebras and related structures.
    \newblock {\em Asian-European Journal of Mathematics}, 16:2350022, 2023.
    
    \bibitem{Irvin}
    I.~R. Hentzel and L.~A. Peresi.
    \newblock Special identities for {B}ol algebras.
    \newblock {\em Linear Algebra and its Applications}, 436:2315–2330, 2012.
    
    \bibitem{Jac}
    N.~Jacobson.
    \newblock {L}ie and Jordan triple systems.
    \newblock {\em American Journal of Mathematics}, 71(1):149–-170, 1949.
    
    \bibitem{k23}
    I.~Kaygorodov.
    \newblock Non-associative algebraic structures: classification and structure.
    \newblock {\em Communications in Mathematics}, 32(3):1--62, 2024.
    
    \bibitem{MS}
    I.~Kaygorodov, M.~Khrypchenko, and P.~P{á}ez-Guill{á}n.
    \newblock The geometric classification of non-associative algebras: a survey.
    \newblock {\em Communications in Mathematics}, 32(2):185--284, 2024.
    
    \bibitem{Khr}
    M.~Khrypchenko.
    \newblock $\sigma$-matching and interchangeable structures on certain
      associative algebras.
    \newblock {\em Communications in Mathematics}, 33(3):6, 2025.
    
    \bibitem{Kik}
    M.~Kikkawa.
    \newblock Geometry of homogeneous {L}ie loops.
    \newblock {\em Hiroshima Mathematical Journal}, 5:141--179, 1975.
    
    \bibitem{LY}
    M.~K. Kinyon and A.~Weinstein.
    \newblock {L}eibniz algebras, {C}ourant algebroids, and multiplications on
      reductive homogeneous spaces.
    \newblock {\em American Journal of Mathematics}, 123(3):525--550, 2001.
    
    \bibitem{Kosmann}
    Y.~Kosmann-Schwarzbach and F.~Magri.
    \newblock {P}oisson-{N}ijenhuis structures.
    \newblock {\em Annales de l'institut Henri Poincar{é} (A) 
      Physique th{é}orique}, 53(1):35--81, 1990.
    
    \bibitem{Kuzmin}
    E.~N. Kuz’min and O.~Za{\u\i}di.
    \newblock Solvable and semisimple {B}ol algebras.
    \newblock {\em Algebra Logic}, 32:361–-371, 1993.
    
    \bibitem{Comp}
    M.~Ladra, B.~Leite~da Cunha, and S.~A. Lopes.
    \newblock A classification of nilpotent compatible {L}ie algebras.
    \newblock {\em Rendiconti del Circolo Matematico di Palermo Series 2}, 70(74),
      2025.
    
    \bibitem{lsb20}
    J.~Liu, Y.~Sheng, and C.~Bai.
    \newblock $F$-manifold algebras and deformation quantization via pre-{L}ie
      algebras.
    \newblock {\em Journal of Algebra}, 559:467--495, 2020.
    
    \bibitem{l24}
    S.~Lopes.
    \newblock Noncommutative algebra and representation theory: symmetry, structure
      \& invariants.
    \newblock {\em Communications in Mathematics}, 32(2):63--117, 2024.
    
    \bibitem{maz79}
    G.~Mazzola.
    \newblock The algebraic and geometric classification of associative algebras of
      dimension five.
    \newblock {\em Manuscripta Mathematica}, 27(1):81--101, 1979.
    
    \bibitem{Mikheev}
    P.~O. Mikheev.
    \newblock Commutator algebras of right alternative algebras, in: Quasigroups
      and their systems.
    \newblock {\em Matematicheskie Issledovaniya}, 113:62--65, 1990.
    
    \bibitem{Nomizu}
    K.~Nomizu.
    \newblock Invariant affine connections on homogeneous spaces.
    \newblock {\em American Journal of Mathematics}, 76:33--65, 1954.
    
    \bibitem{Coml3}
    A.~V. Odesski and V.~V. Sokolov.
    \newblock Compatible {L}ie brackets related to elliptic curve.
    \newblock {\em Journal of Mathematical Physics}, 47(1):013506, 2006.
    
    \bibitem{Perez}
    J.~M. P{é}rez-Izquierdo.
    \newblock An envelope for {B}ol algebras.
    \newblock {\em Journal of Algebra}, 284:480--493, 2005.
    
    \bibitem{Maufang}
    L.~V. Sabinin and P.~O. Mikheev.
    \newblock {\em The theory of smooth {B}ol loops.(Lectures notes)}.
    \newblock Lectures, Univ. Druzhby Narodov, Moscow, 1985.
    
    \bibitem{Sagle2}
    A.~Sagle.
    \newblock On anti-commutative algebras and general {L}ie triple systems.
    \newblock {\em Pacific Journal of Mathematics}, 15:281--291, 1965.
    
    \bibitem{Sagle1}
    A.~Sagle.
    \newblock On simple algebras obtained from homogeneous general {L}ie triple
      systems.
    \newblock {\em Pacific Journal of Mathematics}, 15:1397--1400, 1965.
    
    \bibitem{kms}
    B.~Sartayev.
    \newblock Some generalizations of the variety of transposed Poisson algebras.
    \newblock {\em Communications in Mathematics}, 32(2):55--62, 2024.
    
    \bibitem{LYI}
    K.~Yamaguti.
    \newblock On the {L}ie triple system and its generalization.
    \newblock {\em Journal of Science of the Hiroshima University, Series A},
      21:2212--2231, 1957/1958.
    
    \bibitem{Zhang}
    T.~Zhang and J.~Li.
    \newblock Deformations and extensions of {L}ie-{Y}amaguti algebras.
    \newblock {\em Linear and
      multilinear algebra},
      63(11):2212--2231, 2015.
    
    \end{thebibliography}
}
\end{document}